\documentclass[twoside]{amsart}

\usepackage{amsmath,amsfonts,amsthm,mathrsfs}
\usepackage{amssymb}

\usepackage[ocgcolorlinks,unicode,bookmarks]{hyperref}
\usepackage[usenames,dvipsnames]{xcolor}
\hypersetup{colorlinks=true,citecolor=NavyBlue,linkcolor=BrickRed,urlcolor=Green} 

\usepackage{expl3}

\ExplSyntaxOn
\prop_new:N \g_cite_map_prop
\tl_new:N \l_citekey_result_tl

\cs_new:Npn \mapcitekey #1#2 {
  \clist_map_inline:nn {#2}
       {  \prop_gput:Nnn  \g_cite_map_prop  {##1} {#1}   }
}

\cs_new:Npn \getcitekey #1 {
   \prop_get:NoN \g_cite_map_prop{#1}  \l_citekey_result_tl
   \quark_if_no_value:NF \l_citekey_result_tl
       {  \tl_set_eq:NN #1  \l_citekey_result_tl  }
}

\cs_new:Npn \showcitekeymaps {\prop_show:N  \g_cite_map_prop }
\ExplSyntaxOff

\usepackage{etoolbox}
\makeatletter
\patchcmd{\@citex}{\if@filesw}{\getcitekey\@citeb \if@filesw}%
    {\typeout{*** SUCCESS ***}}{\typeout{*** FAIL ***}}
\patchcmd{\nocite}{\if@filesw}{\getcitekey\@citeb \if@filesw}%
    {\typeout{*** SUCCESS ***}}{\typeout{*** FAIL ***}}
\makeatother

\usepackage{enumitem}

\newenvironment{aenumerate}{%
	\begin{enumerate}[label=(\alph{*}), ref=(\alph{*})]
}{%
	\end{enumerate}%
}

\usepackage{chngcntr}

\usepackage{tikz}

\usepackage{tikz-cd}
\tikzset{commutative diagrams/arrow style=math font}

\newcommand{\Dmod}{\mathscr{D}}
\newcommand{\Mmod}{\mathcal{M}}

\newcommand{\shT}{\mathscr{T}}

\newcommand{\derR}{\mathbf{R}}

\newcommand{\decal}[1]{\lbrack #1 \rbrack}

\newcommand{\tensor}{\otimes}

\newcommand{\shHom}{\mathcal{H}\hspace{-1pt}\mathit{om}}

\newcommand{\ZZ}{\mathbb{Z}}

\newcommand{\CC}{\mathbb{C}}

\newcommand{\PP}{\mathbb{P}}

\newcommand{\menge}[2]{\bigl\{ \thinspace #1 \thinspace\thinspace \big\vert%
\thinspace\thinspace #2 \thinspace \bigr\}}

\DeclareMathOperator{\Res}{Res}

\DeclareMathOperator{\id}{id}

\DeclareMathOperator{\gr}{gr}
\DeclareMathOperator{\DR}{DR}

\DeclareMathOperator{\Ch}{Ch}

\newcommand{\define}[1]{\emph{#1}}

\newcommand{\lie}[2]{\lbrack #1, #2 \rbrack}

\newcommand{\shf}[1]{\mathscr{#1}}
\newcommand{\OX}{\shf{O}_X}
\newcommand{\OmX}{\Omega_X}

\newcommand{\restr}[1]{\big\vert_{#1}}

\def\overbar#1#2#3{{%
	\setbox0=\hbox{\displaystyle{#1}}%
	\dimen0=\wd0
	\advance\dimen0 by -#2 
	\vbox {\nointerlineskip \moveright #3 \vbox{\hrule height 0.3pt width \dimen0}%
		\nointerlineskip \vskip 1.5pt \box0}%
}}

\newcommand{\into}{\hookrightarrow}

\newcommand{\class}[1]{\lbrack #1 \rbrack}

\newcommand{\pil}{\pi_{\ast}}
\newcommand{\jl}{j_{\ast}}

\newcommand{\ju}{j^{\ast}}
\newcommand{\fu}{f^{\ast}}
\newcommand{\fl}{f_{\ast}}

\newcommand{\pu}{p^{\ast}}
\newcommand{\pl}{p_{\ast}}

\newcommand{\tl}{t_{\ast}}

\newcommand{\piu}{\pi^{\ast}}

\newcommand{\DD}{\mathbb{D}}

\newcommand{\shO}{\shf{O}}

\makeatletter
\let\@@seccntformat\@seccntformat
\renewcommand*{\@seccntformat}[1]{%
  \expandafter\ifx\csname @seccntformat@#1\endcsname\relax
    \expandafter\@@seccntformat
  \else
    \expandafter
      \csname @seccntformat@#1\expandafter\endcsname
  \fi
    {#1}%
}
\newcommand*{\@seccntformat@subsection}[1]{%
  \textbf{\csname the#1\endcsname.}
}
\makeatother

\makeatletter
\let\@paragraph\paragraph
\renewcommand*{\paragraph}[1]{%
	\vspace{0.3\baselineskip}%
	\@paragraph{\textit{#1}}%
}
\makeatother

\counterwithin{equation}{subsection}
\counterwithout{subsection}{section}
\counterwithin{figure}{subsection}

\newtheorem{theorem}[equation]{Theorem}
\newtheorem*{theorem*}{Theorem}
\newtheorem{lemma}[equation]{Lemma}
\newtheorem*{lemma*}{Lemma}

\newtheorem{proposition}[equation]{Proposition}
\newtheorem*{proposition*}{Proposition}

\newtheorem*{conjecture*}{Conjecture}

\theoremstyle{definition}
\newtheorem{definition}[equation]{Definition}
\newtheorem*{definition*}{Definition}
\theoremstyle{remark}

\newtheorem{example}[equation]{Example}
\newtheorem*{example*}{Example}
\newtheorem*{problem*}{Problem}
\newtheorem*{note}{Note}

\theoremstyle{plain}

\newcommand{\theoremref}[1]{\hyperref[#1]{Theorem~\ref*{#1}}}
\newcommand{\lemmaref}[1]{\hyperref[#1]{Lemma~\ref*{#1}}}
\newcommand{\definitionref}[1]{\hyperref[#1]{Definition~\ref*{#1}}}
\newcommand{\propositionref}[1]{\hyperref[#1]{Proposition~\ref*{#1}}}
\newcommand{\conjectureref}[1]{\hyperref[#1]{Conjecture~\ref*{#1}}}
\newcommand{\corollaryref}[1]{\hyperref[#1]{Corollary~\ref*{#1}}}
\newcommand{\exampleref}[1]{\hyperref[#1]{Example~\ref*{#1}}}
\newcommand{\exerciseref}[1]{\hyperref[#1]{Exercise~\ref*{#1}}}

\makeatletter
\let\old@caption\caption
\renewcommand*{\caption}[1]{%
	\setcounter{figure}{\value{equation}}%
	\stepcounter{equation}%
	\old@caption{#1}\relax%
}
\makeatother

\newcounter{intro}

\newtheorem{intro-conjecture}[intro]{Conjecture}
\newtheorem{intro-corollary}[intro]{Corollary}
\newtheorem{intro-theorem}[intro]{Theorem}

\newcommand{\OmY}{\Omega_Y}
\newcommand{\OY}{\shO_Y}

\newcommand{\df}{\mathit{df}}

\newcommand{\parref}[1]{\hyperref[#1]{\S\ref*{#1}}}
\newcommand{\chapref}[1]{\hyperref[#1]{Chapter~\ref*{#1}}}

\makeatletter
\newcommand*\if@single[3]{%
  \setbox0\hbox{${\mathaccent"0362{#1}}^H$}%
  \setbox2\hbox{${\mathaccent"0362{\kern0pt#1}}^H$}%
  \ifdim\ht0=\ht2 #3\else #2\fi
  }
\newcommand*\rel@kern[1]{\kern#1\dimexpr\macc@kerna}
\newcommand*\widebar[1]{\@ifnextchar^{{\wide@bar{#1}{0}}}{\wide@bar{#1}{1}}}
\newcommand*\wide@bar[2]{\if@single{#1}{\wide@bar@{#1}{#2}{1}}{\wide@bar@{#1}{#2}{2}}}
\newcommand*\wide@bar@[3]{%
  \begingroup
  \def\mathaccent##1##2{%
    \if#32 \let\macc@nucleus\first@char \fi
    \setbox\z@\hbox{$\macc@style{\macc@nucleus}_{}$}%
    \setbox\tw@\hbox{$\macc@style{\macc@nucleus}{}_{}$}%
    \dimen@\wd\tw@
    \advance\dimen@-\wd\z@
    \divide\dimen@ 3
    \@tempdima\wd\tw@
    \advance\@tempdima-\scriptspace
    \divide\@tempdima 10
    \advance\dimen@-\@tempdima
    \ifdim\dimen@>\z@ \dimen@0pt\fi
    \rel@kern{0.6}\kern-\dimen@
    \if#31
      \overline{\rel@kern{-0.6}\kern\dimen@\macc@nucleus\rel@kern{0.4}\kern\dimen@}%
      \advance\dimen@0.4\dimexpr\macc@kerna
      \let\final@kern#2%
      \ifdim\dimen@<\z@ \let\final@kern1\fi
      \if\final@kern1 \kern-\dimen@\fi
    \else
      \overline{\rel@kern{-0.6}\kern\dimen@#1}%
    \fi
  }%
  \macc@depth\@ne
  \let\math@bgroup\@empty \let\math@egroup\macc@set@skewchar
  \mathsurround\z@ \frozen@everymath{\mathgroup\macc@group\relax}%
  \macc@set@skewchar\relax
  \let\mathaccentV\macc@nested@a
  \if#31
    \macc@nested@a\relax111{#1}%
  \else
    \def\gobble@till@marker##1\endmarker{}%
    \futurelet\first@char\gobble@till@marker#1\endmarker
    \ifcat\noexpand\first@char A\else
      \def\first@char{}%
    \fi
    \macc@nested@a\relax111{\first@char}%
  \fi
  \endgroup
}
\makeatother

\mapcitekey{Saito:HodgeModules}{Saito-HM}
\mapcitekey{Saito:MixedHodgeModules}{Saito-MHM}
\mapcitekey{Durfee:Neighborhoods}{Durfee}
\mapcitekey{Ramanujam:KodairaVanishing}{Ramanujam}
\mapcitekey{Esnault+Viehweg:VanishingTheorems}{EV}
\mapcitekey{Esnault+Viehweg:LogarithmicDeRham}{EV-log}
\mapcitekey{Saito:DModules}{Saito-an}
\mapcitekey{Kodaira:DifferentialGeometricMethod}{Kodaira}
\mapcitekey{Deligne+Illusie}{DI}
\mapcitekey{Saito:Theory}{Saito-th}
\mapcitekey{Saito:Kollar}{Saito-K}
\mapcitekey{Popa:SaitoVanishing}{Popa}
\mapcitekey{Schnell:sanya}{sanya}

\newcommand{\HM}[2]{\operatorname{HM}(#1, #2)}

\newcommand{\HMp}[2]{\operatorname{HM}(#1, #2)}
\newcommand{\HMZp}[3]{\operatorname{HM}_{#1}(#2, #3)}
\newcommand{\MHMp}[1]{\operatorname{MHM}(#1)}
\newcommand{\MHMps}[2]{\operatorname{MHM}(#1, #2)}
\newcommand{\MHMpS}[2]{\operatorname{MHM} \bigl( #1, #2 \bigr)}

\newcommand{\Dbcoh}{\mathrm{D}_{\mathit{coh}}^{\mathit{b}}}
\newcommand{\OP}{\shO_P}

\newcommand{\OXst}{\OX^{\times}}
\newcommand{\OZst}{\OZ^{\times}}
\newcommand{\OZ}{\shO_Z}

\newcommand{\OmD}{\Omega_D}
\newcommand{\Kt}{\tilde{K}}
\newcommand{\deltat}{\tilde{\delta}}

\newcommand{\shI}{\mathscr{I}}
\newcommand{\omX}{\omega_X}
\renewcommand{\DD}{\mathbf{D}}
\newcommand{\dx}{\mathit{dx}}
\newcommand{\shK}{\mathcal{K}}

\begin{document}

\title{On Saito's vanishing theorem}
\author{Christian Schnell}
\address{%
Department of Mathematics\\
Stony Brook University\\
Stony Brook, NY 11794-3651}
\email{cschnell@math.sunysb.edu}

\begin{abstract}
We reprove Saito's vanishing theorem for mixed Hodge modules by the method of Esnault
and Viehweg. The main idea is to exploit the strictness of direct images on certain
branched coverings.
\end{abstract}
\date{\today}
\maketitle

\section{Overview}

\subsection{Introduction}

The Kodaira vanishing theorem is one of the most useful results in algebraic
geometry. Besides the original differential-geometric proof by Kodaira \cite{Kodaira}
and the famous algebraic proof by Deligne and Illusie \cite{DI}, there are at
least two other proofs that are based on Hodge theory. One is due to
Ramanujam \cite{Ramanujam}, and uses the weak Lefschetz theorem; the other is due to
Esnault and Viehweg \cite{EV-log}, and uses branched coverings and the degeneration
of the Hodge-de Rham spectral sequence.

Saito's vanishing theorem \cite[\S2.g]{Saito-MHM} is a generalization of Kodaira's theorem to
mixed Hodge modules; it contains as special cases several other results, such as
Koll\'ar's vanishing theorem for higher direct images of dualizing sheaves. More
precisely, Saito uses Artin's vanishing theorem for perverse sheaves on affine
varieties to obtain a vanishing theorem for the graded quotients of the de Rham
complex of any graded-polarizable mixed Hodge module; his proof is therefore a
distant cousin of Ramanujam's. In this paper, we show that Saito's theorem can
also be proved by the method of Esnault and Viehweg: the key point is to exploit the
strictness of direct images on certain branched
coverings. The argument is perhaps less elegant than Saito's, but it has
two advantages: 
\begin{enumerate} 
\item The vanishing theorem follows from results about polarizable Hodge
modules, without appealing to vanishing theorems for perverse sheaves.
\item The result can be stated and proved entirely in terms of pure Hodge
modules, without the need for using mixed Hodge modules.
\end{enumerate}
Since mixed Hodge modules are more complicated than pure ones, the second point may
be useful to someone who is trying to understand Saito's vanishing theorem
with a minimum of theoretical background. Those who are interested in the original
proof can also consult Popa's expository paper \cite{Popa}.

\subsection{Statement of the result}

We will first state the vanishing theorem for mixed Hodge modules, because this is the
version that Saito gives; but in fact, the general case follows very easily from
the special case of pure Hodge modules.

Let $Z$ be a reduced projective algebraic variety. We denote by $\MHMp{Z}$ the
abelian category of graded-polarizable mixed Hodge modules on $Z$. It is defined by
embedding $Z$ into a complex manifold $X$, for example into complex projective space,
and then looking at all graded-polarizable mixed Hodge modules on $X$ whose support
is contained in $Z$. One can show that this definition is independent of the choice of
embedding; for the convenience of the reader, an outline of the proof is included in
\parref{par:MHM} below. Given $M \in \MHMp{Z}$, we write $(\Mmod, F_{\bullet} \Mmod)$
for the underlying filtered $\Dmod$-module: $\Mmod$ is a regular holonomic left
$\Dmod$-module on $X$ whose support is contained in $Z$, and $F_{\bullet} \Mmod$ is a
good filtration by coherent $\OX$-modules. We set $n = \dim X$, and denote by
\[
	\DR(\Mmod) = \Bigl\lbrack
		\Mmod \to \OmX^1 \tensor \Mmod \to \dotsb \to \OmX^n \tensor \Mmod
	\Bigr\rbrack \decal{n}
\]
the de Rham complex of $\Mmod$; by a theorem of Kashiwara, it is a
perverse sheaf on $X$ with support in $Z$. The complex $\DR(\Mmod)$ is naturally
filtered by the family of subcomplexes
\[
	F_p \DR(\Mmod) = \Bigl\lbrack
		F_p \Mmod \to \OmX^1 \tensor F_{p+1} \Mmod \to \dotsb 
			\to \OmX^n \tensor F_{p+n} \Mmod
	\Bigr\rbrack \decal{n},
\]
and one can use the properties of mixed Hodge modules to show that each
\[
	\gr_p^F \DR(\Mmod) = \Bigl\lbrack
		\gr_p^F \Mmod \to \OmX^1 \tensor \gr_{p+1}^F \Mmod \to \dotsb 
			\to \OmX^n \tensor \gr_{p+n}^F \Mmod
	\Bigr\rbrack \decal{n}
\]
is a well-defined complex of coherent $\OZ$-modules, whose isomorphism class in the
derived category $\Dbcoh(\OZ)$ does not depend on the choice of ambient complex
manifold $X$. For the convenience of the reader, the argument is recalled in
\lemmaref{lem:subquotient}.

As one of the first applications of his theory of mixed Hodge modules, Saito
proved the following vanishing theorem for those complexes \cite[\S2.g]{Saito-MHM}.

\begin{theorem} \label{thm:Saito}
Let $M \in \MHMp{Z}$ be a graded-polarizable mixed Hodge module on a reduced
projective variety $Z$. If $L$ is an ample line bundle on $Z$, one has
\begin{align*}
	H^i \bigl( Z, \gr_p^F \DR(\Mmod) \tensor L \bigr) = 0 
		\quad \text{for $i > 0$ and $p \in \ZZ$,} \\
	H^i \bigl( Z, \gr_p^F \DR(\Mmod) \tensor L^{-1} \bigr) = 0 
		\quad \text{for $i < 0$ and $p \in \ZZ$.} 
\end{align*}
\end{theorem}

Note that Kodaira's vanishing theorem is a special case: the pair
$(\OX, F_{\bullet} \OX)$, with $\gr_p^F \OX = 0$ for $p \neq 0$, is part of a
polarizable Hodge module, and one has
\[
	\omX = \gr_{-n}^F \DR(\OX) \quad \text{and} \quad
		\OX \decal{n} = \gr_0^F \DR(\OX).
\]
Although \theoremref{thm:Saito} is stated in terms of mixed Hodge modules, it is
really a result about pure ones; we will see below that the same is true for
the proof.  

\subsection{Reduction to the pure case}

We now explain how to obtain \theoremref{thm:Saito} from a statement about pure Hodge
modules. For a reduced and irreducible projective variety $Z$, we denote by
$\HMZp{Z}{Z}{w}$ the abelian category of polarizable Hodge modules of weight $w$ with
strict support $Z$; the precise definition is
\[
	\HMZp{Z}{Z}{w} = \HMZp{Z}{X}{w},
\]
where $X$ is a complex manifold containing $Z$, and where $M \in \HMp{X}{w}$ belongs
to $\HMZp{Z}{X}{w}$ iff the support of every nonzero subobject or quotient object of
$M$ is equal to $Z$. As before, one can show that this category does not depend on the
choice of embedding; in fact, an important result by Saito
\cite[Theorem~3.21]{Saito-MHM} says that $\HMZp{Z}{Z}{w}$ is equivalent to the
category of generically defined polarizable variations of Hodge
structure of weight $w - \dim Z$.

\begin{theorem} \label{thm:Saito-pure}
Let $Z$ be a reduced and irreducible projective variety, and let $M \in \HMZp{Z}{Z}{w}$
be a polarizable Hodge module with strict support $Z$. Then one has
\begin{align*}
	H^i \bigl( Z, \gr_p^F \DR(\Mmod) \tensor L \bigr) = 0 
		\quad \text{for $i > 0$ and $p \in \ZZ$,} \\
	H^i \bigl( Z, \gr_p^F \DR(\Mmod) \tensor L^{-1} \bigr) = 0 
		\quad \text{for $i < 0$ and $p \in \ZZ$,} 
\end{align*}
where $L$ is any ample line bundle on $Z$.
\end{theorem}

It is easy to deduce \theoremref{thm:Saito} from this special case. Suppose first
that $Z$ is a reduced projective variety, and that $M \in \HMp{Z}{w}$ is a
polarizable Hodge module of weight $w$. Then $M$ admits a decomposition by
strict support, and because the vanishing theorem is true for each summand by
\theoremref{thm:Saito-pure}, it is true for $M$ as well. To deal with the general
case, recall that every $M \in \MHMp{Z}$ has a finite weight filtration $W_{\bullet}
M$ with the property that $\gr_w^W M \in \HMp{Z}{w}$; because the functor $\gr_p^F
\DR$ is exact, we obtain the vanishing theorem for arbitrary graded-polarizable mixed
Hodge modules.

\subsection{Idea of the proof}

To prove \theoremref{thm:Saito-pure}, we shall use a method invented by Esnault and
Viehweg. The general idea, explained for example in \cite[\S1]{EV}, is to deduce
vanishing theorems from the $E_1$-degeneration of certain spectral sequences.

As a motivation for what follows, let us briefly recall how Esnault and Viehweg prove
the Kodaira vanishing theorem. Let $L$ be an ample line bundle on a smooth projective
variety $X$. For sufficiently large $N$, the line bundle $L^N$ becomes very ample,
and we can find a smooth divisor $D \subseteq X$ with $L^N \simeq \OX(D)$. Such a
divisor determines a branched covering $\pi \colon Y \to X$ (see
\parref{par:coverings}), and one can show that
\[
	\pil \OY \simeq \OX \oplus \bigoplus_{i=1}^{N-1} L^{-i} \quad \text{and} \quad
		\pil \OmY^1 \simeq \OmX^1 \oplus 
			\bigoplus_{i=1}^{N-1} \OmX^1(\log D) \tensor L^{-i}.
\]
Now $Y$ is again a smooth projective variety, and so its Hodge-de Rham spectral
sequence degenerates at $E_1$; in particular, the mapping $d \colon H^i(Y, \OY) \to
H^i(Y, \OmY^1)$ is equal to zero. From this, one can deduce that the restriction mapping
\[
	H^i(X, L^{-1}) \to H^i(D, \shO_D \tensor L^{-1})
\]
is also equal to zero: the key point is that $d \colon \OY \to \OmY^1$ induces a
$\CC$-linear mapping $L^{-1} \to \OmX^1(\log D) \tensor L^{-1}$, whose composition
with the residue mapping is, up to a constant factor, equal to the $\OX$-linear mapping
$L^{-1} \to \shO_D \tensor L^{-1}$. Consequently, 
\[
	H^i(X, L^{-N-1}) \to H^i(X, L^{-1})
\]
must be surjective; because of Serre duality,
\[
	H^i(X, \omX \tensor L) \to H^i(X, \omX \tensor L^{N+1})
\]
must be injective. But now we can kill the right-hand side by taking $N \gg 0$, and
so we get the vanishing of $H^i(X, \omX \tensor L)$ for $i > 0$.

The proof of \theoremref{thm:Saito-pure} follows the same path. Since $Z$ may be
singular, we first extend the line bundle $L$ to a small open neighborhood $X$ in
some projective embedding of $Z$. We then take a sufficiently generic branched
covering $\pi \colon Y \to X$, and use the strictness of direct images for
polarizable Hodge modules to prove that
\[
	H^i \Bigl( Z, L \tensor \gr_p^F \DR(\Mmod) \Bigr) \to
	H^i \Bigl( Z, L^{N+1} \tensor \gr_p^F \DR(\Mmod) \Bigr)
\]
must be injective for $i > 0$. Because the complex $\gr_p^F \DR(\Mmod)$ is
concentrated in non-positive degrees, we can again kill the right-hand side by taking
$N \gg 0$. This proves half of \theoremref{thm:Saito-pure}; the other half follows by
Sere duality, because the de Rham complex is compatible with the duality functor (see
\lemmaref{lem:polarization}).

\subsection{Note to the reader}

I wrote this paper for those who already know the definitions and basic results from
the theory of polarizable Hodge modules \cite{Saito-HM}. If you are not familiar with
Saito's theory, but nevertheless interested in the proof of the vanishing theorem, I
would recommend taking a look at Saito's nicely-written survey article
\cite{Saito-th} or at the more recent \cite{sanya}. Two of the results that we need
-- about Hodge modules on singular varieties and about non-characteristic inverse
images -- are distributed among several of Saito's papers; in the interest of
readability, I therefore decided to include an outline of their proof.  

Please note that I chose to use left $\Dmod$-modules throughout: this
is more convenient when working with inverse images and simplifies certain
arguments with differential forms. However, because Saito uses right $\Dmod$-modules,
a little bit of translation is needed when looking up results in
\cite{Saito-HM,Saito-MHM}. The rules are as follows. Suppose that $(\Mmod,
F_{\bullet} \Mmod)$ is a filtered left $\Dmod$-module on an $n$-dimensional complex
manifold $X$. Then the associated filtered right $\Dmod$-module is
\[
	\Bigl( \omX \tensor_{\OX} \Mmod, \, \omX \tensor_{\OX} F_{\bullet+n} \Mmod \Bigr),
\]
with $\Dmod$-module structure given by $(\omega \tensor m) \cdot \xi =
(\omega \cdot \xi) \tensor m - \omega \tensor (\xi \cdot m)$, where $\omega$, $m$,
and $\xi$ are sections of $\omX$, $\Mmod$, and the tangent sheaf $\shT_X$,
respectively.

The conventions for indexing the $V\!$-filtration are also different for left and right
$\Dmod$-modules. In the case of $\Mmod$, the rational $V\!$-filtration along $t = 0$ is
a decreasing filtration $V^{\bullet} \Mmod$ with the property that $t \partial_t -
\alpha$ acts nilpotently on $\gr_V^{\alpha} \Mmod$. The corresponding filtration on
$\omX \tensor \Mmod$ \cite[D\'efinition~3.1.1]{Saito-HM} is the increasing filtration
\[
	V_{\bullet} \bigl( \omX \tensor_{\OX} \Mmod \bigr)
		= \omX \tensor_{\OX} V^{-\bullet-1} \Mmod;
\]
the change is needed to keep the (right) action of $t \partial_t - \alpha$ on
$\gr_{\alpha}^V(\omX \tensor \Mmod)$ nilpotent.

\subsection{Acknowledgements}

I thank Mihnea Popa for many useful conversations about mixed Hodge modules and
vanishing theorems. During the preparation of this paper, I have been supported in
part by NSF-grant DMS-1331641. 

\section{Some background}

\subsection{Mixed Hodge modules on singular varieties}
\label{par:MHM}

Since \theoremref{thm:Saito} is stated for mixed Hodge modules on projective
varieties, it may be helpful to review the definition of the category $\MHMp{Z}$ in
the case where $Z$ is a possibly singular projective algebraic variety. The idea is
to embed $Z$ into a complex manifold $X$ (such as projective space), and then to
define
\[
	\MHMp{Z} \subseteq \MHMp{X}
\]
as the full subcategory of all graded-polarizable mixed Hodge modules on $X$ whose
support is contained in $Z$. To make this definition meaningful, one has to show that
the resulting category does not depend on the embedding; we shall explain below how
this is done.

\begin{note}
The same definition works for any analytic space that can be embedded into a
complex manifold. On an arbitrary analytic space $Z$, such embeddings may only exist
locally, and so one has to cover $Z$ by embeddable open subsets and work with
collections of mixed Hodge modules on the open sets that are compatible on
intersections. This idea is developed in \cite{Saito-an}. 
\end{note}

From now on, let $Z$ be a reduced projective algebraic variety. Given an embedding $i
\colon Z \into X$ into a complex manifold, we consider the full subcategory
\[
	\MHMps{X}{i} \subseteq \MHMp{X}
\]
of all graded-polarizable mixed Hodge modules on $X$ whose support is contained in
the image of $Z$. Obviously, we can always take $X$ to be projective space of some
dimension; but for the proof of \theoremref{thm:Saito}, it will be useful to allow
other complex manifolds, too. It is easy to see that $\MHMps{X}{i}$ is an abelian
category. The main result is that this category does not depend on the choice of
embedding.

\begin{proposition} \label{prop:independence}
Given two embeddings $i \colon Z \into X$ and $j \colon Z \into Y$, 
one has a canonical equivalence of categories between $\MHMps{X}{i}$ and
$\MHMps{Y}{j}$.
\end{proposition}

The tool for proving this is the following version of Kashiwara's equivalence
for mixed Hodge modules.

\begin{proposition} \label{prop:Kashiwara}
Let $f \colon X \into Y$ be a closed embedding between two complex manifolds. For any
closed analytic subspace $i \colon Z \into X$, the direct image functor
\[
	\fl \colon \MHMp{X} \to \MHMp{Y}
\]
induces an equivalence of categories between $\MHMps{X}{i}$ and $\MHMps{Y}{f \circ i}$.
\end{proposition}

\begin{proof}
This is proved in \cite[Lemme~5.1.9]{Saito-HM} for pure Hodge modules, and asserted in
\cite[2.17.5]{Saito-MHM} for mixed ones. The essential point is to show that the
underlying filtered $\Dmod$-module $(\Mmod, F_{\bullet} \Mmod)$ of a mixed Hodge
module $M \in \MHMps{Y}{f \circ i}$ comes from $X$. For $\Mmod$, this follows
from a more general result for coherent $\Dmod$-modules in Kashiwara's thesis; in
order to deal with the filtration $F_{\bullet} \Mmod$, one has to use one of the
axioms characterizing mixed Hodge modules \cite[Proposition~3.2.2]{Saito-HM}.
\end{proof}

Now let us prove \propositionref{prop:independence}. Since we cannot directly compare
$X$ and $Y$, we use the product embedding $(i,j) \colon Z \into X \times Y$, as in
the following diagram:
\begin{equation} \label{eq:embeddings}
\begin{tikzcd}
Z \arrow{dr}{(i,j)} \arrow[bend left=25]{drr}{j} \arrow[bend right=30]{ddr}{i} \\
& X \times Y \dar{p} \rar{q} & Y \\
& X
\end{tikzcd}
\end{equation}
Because the situation is symmetric, it suffices to show that the direct image functor
\[
	\pl \colon \MHMpS{X \times Y}{(i,j)} \to \MHMps{X}{i}
\]
is an equivalence of categories. Note that $\pl$ is obviously faithful: in fact, this
is true for the underlying perverse sheaves because $p$ is an isomorphism on the
image of $Z$, and the functor from mixed Hodge modules to perverse sheaves is
faithful. So the issue is to show that $\pl$ is essentially surjective.

Let $M \in \MHMp{X}$ be a graded-polarizable mixed Hodge module whose support is
contained in $i(Z)$. To construct from $M$ an object on $X \times Y$, we use the
existence of good local sections for $p$. More precisely, for every point of $Z$,
there is an open neighborhood $U \subseteq X$ and a holomorphic mapping $f \colon U
\to Y$ such that $f \circ i = j$; this follows from the basic properties of
holomorphic functions. Now
\[
	(\id, f) \colon U \into U \times Y 
\]
is a closed embedding with the property that $(\id, f) \circ i = (i,j)$, and so $(\id,
f)_{\ast} M$ is a graded-polarizable mixed Hodge module on $U \times Y$
whose support is contained in the image of $(i,j)$. If we choose a different
holomorphic mapping $f' \colon U \to Y$, then $(\id, f)_{\ast} M$ and $(\id,
f')_{\ast} M$ are canonically isomorphic by virtue of \propositionref{prop:Kashiwara}. 
This fact allows us to glue the local objects together into a well-defined object of
$\MHMpS{X \times Y}{(i,j)}$; it is clear from the construction that its image under
$\pl$ is isomorphic to the original mixed Hodge module $M$.

\subsection{Subquotients of the de Rham complex} 
\label{par:DR}

In this section, we collect a few general results about the graded quotients of the
de Rham complex. Let $X$ be a complex manifold, and let $(\Mmod, F_{\bullet} \Mmod)$
be a filtered $\Dmod$-module on $X$. We begin with a more careful local description
of the differentials in the complex
\[
	\DR(\Mmod) = \Bigl\lbrack
		\Mmod \to \OmX^1 \tensor \Mmod \to \dotsb \to \OmX^n \tensor \Mmod
	\Bigr\rbrack \decal{n}.
\]
Let $x_1, \dotsc, x_n$ be local holomorphic coordinates on $X$. Then the
differentials 
\[
	\nabla \colon \OmX^k \tensor \Mmod \to \OmX^{k+1} \tensor \Mmod
\]
in the de Rham complex are given by the formula
\begin{equation} \label{eq:DR-local}
	\nabla(\alpha \tensor m) = (-1)^n d\alpha \tensor m + (-1)^n 
		\sum_{i=1}^n (\dx_i \wedge \alpha) \tensor \frac{\partial}{\partial x_i} m.
\end{equation}
The extra factor of $(-1)^n$ is due to the shift in the definition of $\DR(\Mmod)$;
it is part of a consistent set of sign conventions. Because $F_{\bullet} \Mmod$ is a
good filtration, it is obvious from this description that each 
\begin{equation} \label{eq:DR-F}
	F_p \DR(\Mmod) = \Bigl\lbrack
		F_p \Mmod \to \OmX^1 \tensor F_{p+1} \Mmod \to \dotsb 
			\to \OmX^n \tensor F_{p+n} \Mmod
	\Bigr\rbrack \decal{n}
\end{equation}
is a subcomplex. When we go to one of the graded quotients $\gr_p^F \DR(\Mmod)$, we
obtain the following formula for the differentials:
\begin{align*}
	\OmX^k \tensor \gr_{p+k}^F \Mmod \to \OmX^{k+1} \tensor \gr_{p+k+1}^F \Mmod, \quad
	\alpha \tensor m \mapsto (-1)^n 
		\sum_{i=1}^n (\dx_i \wedge \alpha) \tensor \frac{\partial}{\partial x_i} m.
\end{align*}

Now let us consider the case where $(\Mmod, F_{\bullet} \Mmod)$ is part of a mixed
Hodge module on $X$. In the case where the support of $\Mmod$ is contained in an
analytic subset $Z$, the properties of mixed Hodge modules imply that each
$\gr_p^F \DR(\Mmod)$ is actually a complex of coherent $\OZ$-modules.

\begin{lemma} \label{lem:subquotient}
Let $M \in \MHMp{X}$ be a mixed Hodge module on a complex manifold $X$. If the
support of $M$ is contained in an analytic subset $Z \subseteq X$, then each
\[
	\gr_p^F \DR(\Mmod) = \Bigl\lbrack
		\gr_p^F \Mmod \to \OmX^1 \tensor \gr_{p+1}^F \Mmod \to \dotsb 
			\to \OmX^n \tensor \gr_{p+n}^F \Mmod
	\Bigr\rbrack \decal{n}
\]
is a well-defined complex of coherent $\OZ$-modules; its isomorphism class in
$\Dbcoh(\OZ)$ is independent of the embedding of $Z$ into a complex manifold.
\end{lemma}

\begin{proof}
We first prove that each $\gr_p^F \Mmod$ is a coherent sheaf on $Z$. Let $f$ be an
arbitrary local section of the ideal sheaf $\shI_Z$; by the definition of mixed Hodge
modules, $(\Mmod, F_{\bullet} \Mmod)$ is quasi-unipotent and regular along $f = 0$. By
\cite[Lemme~4.2.6]{Saito-HM}, this implies that $f \cdot F_p \Mmod \subseteq F_{p-1}
\Mmod$, which means that $f$ annihilates $\gr_p^F \Mmod$.

The independence of the choice of embedding follows from
\propositionref{prop:independence} and the compatibility of the de Rham complex with
direct images. Suppose we have another embedding $j \colon Z \into Y$ into a complex
manifold. As in \eqref{eq:embeddings}, we consider the product embedding $(i,j)
\colon Z \into X \times Y$. By \propositionref{prop:independence}, we have $M \simeq
\pl M'$ for a graded-polarizable mixed Hodge module $M' \in \MHMp{X \times Y}$ whose 
support is contained in $(i,j)(Z)$; as the situation is symmetric, it suffices to
prove that 
\[
	\pl \Bigl( \gr_p^F \DR(\Mmod') \Bigr) \simeq 
	\derR \pl \Bigl( \gr_p^F \DR(\Mmod') \Bigr) \simeq \gr_p^F \DR(\Mmod).
\]
But because $p$ is an isomorphism over the image of $Z$, this follows from the
definition of the direct image functor for filtered $\Dmod$-modules \cite[\S2.3.7]{Saito-HM}.  
\end{proof}

Another very useful result is the compatibility of the de Rham complex with the
duality functor. In combination with Serre duality, it can be used to show that the two
assertions for $L$ and $L^{-1}$ in \theoremref{thm:Saito} are equivalent.

\begin{lemma} \label{lem:polarization}
Let $M \in \HMp{X}{w}$ be a polarizable Hodge module on an $n$-dimensional complex
manifold $X$. Then any polarization on $M$ induces an isomorphism
\[
	\derR \shHom_{\OX} \Bigl( \gr_p^F \DR(\Mmod), \omX \decal{n} \Bigr) 
		\simeq \gr_{-p-w}^F \DR(\Mmod).
\]
\end{lemma}

\begin{proof}
Recall that a polarization of $M$ induces an isomorphism $M(w) \simeq \DD M$ with the
dual Hodge module; in particular, the filtered $\Dmod$-module underlying $\DD M$ is
isomorphic to $(\Mmod, F_{\bullet-w} \Mmod)$. The assertion therefore follows from
the compatibility of the filtered de Rham complex with the duality functor
\cite[\S2.4.3]{Saito-HM}. Since the result is not explicitely stated there, we shall
quickly sketch the proof.

The main tool is the equivalence, on the level of derived categories, between
filtered $\Dmod$-modules and filtered differential complexes \cite[\S2.2]{Saito-HM};
under this equivalence, the pair $(\Mmod, F_{\bullet} \Mmod)$ goes to the de Rham
complex $\DR(\Mmod)$, endowed with the filtration in \eqref{eq:DR-F}. Now choose an
injective resolution
\[
	0 \to \omX \to \shK^{-n} \to \dotsb \to \shK^{-1} \to \shK^0 \to 0
\]
by right $\Dmod_X$-modules that are injective as $\OX$-modules; such a resolution
exists because the injective dimension of $\omX$ is equal to $n$. As explained in
\cite[\S2.4.11]{Saito-HM}, the de Rham complex of $\DD M$ is isomorphic, as
a filtered differential complex, to the simple complex associated with the double
complex
\[
	\shHom_{\OX} \bigl( \DR(\Mmod), \shK^{\bullet} \bigr).
\]
Here the filtration on the double complex is given by the rule
\begin{align*}
	F_p \shHom_{\OX}& \bigl( \OmX^{n-i} \tensor \Mmod, \shK^j \bigr) \\
		= &\menge{\phi \colon \OmX^{n-i} \tensor \Mmod \to \shK^j}{%
		\phi \bigl( \OmX^{n-i} \tensor F_{n-i-p-1} \Mmod \bigr) = 0},
\end{align*}
due to the fact that $\gr_p^F \shK^j = 0$ for $p \neq 0$. In particular, we have
\[
	\gr_p^F \shHom_{\OX} \bigl( \OmX^{n-i} \tensor \Mmod, \shK^j \bigr) \simeq 
		\shHom_{\OX} \bigl( \OmX^{n-i} \tensor \gr_{n-i-p}^F \Mmod, \shK^j \bigr).
\]
This gives us a canonical isomorphism in the derived category between the associated
graded of the de Rham complex of $\DD M$ and 
\[
	\derR \shHom_{\OX} \Bigl( \gr_{-\bullet}^F \DR(\Mmod), \omX \decal{n} \Bigr),
\]
using that $\shK^{\bullet}$ is quasi-isomorphic to $\omX \decal{n}$. Together with
the remark about the polarization from above, this implies the asserted isomorphism.
\end{proof}

This result also bounds the range in which the graded quotients of the de Rham
complex are nontrivial. Since $F_p \Mmod = 0$ for $p \ll 0$, it makes sense to define
\begin{equation} \label{eq:pM}
	p(M) = \min \menge{p \in \ZZ}{\gr_p^F \DR(\Mmod) \neq 0}.
\end{equation}
By definition, we have $\gr_p^F \DR(\Mmod) = 0$ for $p < p(M)$;
\lemmaref{lem:polarization} shows that the complex $\gr_p^F \DR(\Mmod)$ is also exact
for $p > -p(M) - w$. Another consequence is that the Grothendieck dual of the
coherent sheaf
\[
	\gr_{p(M)}^F \DR(\Mmod) = \omX \tensor F_{p(M)+n} \Mmod
\]
is isomorphic to the complex $\gr_{-p(M)-w}^F \DR(\Mmod)$.

%
%
%

\subsection{Non-characteristic inverse images}
\label{par:inverse-image}

In this section, we review the construction of inverse images for polarizable Hodge
modules under sufficiently generic morphisms. Let $f \colon Y \to X$ be a holomorphic
mapping between complex manifolds, and let $r = \dim Y - \dim X$ denote its relative
dimension. In this situation, we have the following morphisms between the cotangent
bundles of $X$ and $Y$:
\[
\begin{tikzcd}
Y \times_X T^{\ast} X \dar{p_2} \rar{\df} & T^{\ast} Y \\
T^{\ast} X
\end{tikzcd}
\]
Let $\Mmod$ be a regular holonomic left $\Dmod$-module on $X$, and $F_{\bullet} \Mmod$ a
good filtration by coherent $\OX$-modules. Recall that the characteristic variety
$\Ch(\Mmod) \subseteq T^{\ast} X$ is the support of the coherent sheaf determined by
the coherent $\gr_{\bullet}^F \! \Dmod_X$-module $\gr_{\bullet}^F \Mmod$.
The following definition is a slightly modified version of \cite[\S3.5.1]{Saito-HM}.

\begin{definition}
We say that the morphism $f$ is \define{non-characteristic} for $(\Mmod, F_{\bullet}
\Mmod)$ if the following two conditions are satisfied:
\begin{aenumerate}
\item The restriction of $\df$ to $p_2^{-1} \Ch(\Mmod)$ is a finite mapping.
\item We have $L^i \fu(\gr_p^F \Mmod) = 0$ for every $i < 0$ and every $p \in \ZZ$.
\end{aenumerate}
\end{definition}

The first condition is a transversality property. Since $\Mmod$ is regular holonomic,
one can find a Whitney stratification adapted to it; note that every irreducible
component of $\Ch(\Mmod)$ is the conormal variety of the closure of a stratum. 
Given a point $y \in Y$, let $S \subseteq X$ be the stratum containing $f(y)$; then
we are asking that 
\[
	T_{f(y)} S + \fl \bigl( T_y Y \bigr) = T_{f(y)} X.
\]
In the case where $Y$ is a subvariety of $X$, for example, this is saying that $Y$ is
transverse to every stratum. The second condition, on the other hand, is a kind of
flatness property: we are asking that the higher derived functors are trivial when we
pull back the $\OX$-modules $\gr_p^F \Mmod$. 

\begin{example}
Smooth morphisms are always non-characteristic.
\end{example}

The point of the two conditions is that the naive pullback $\fu \Mmod$ is again a
regular holonomic $\Dmod$-module on $Y$, and that the filtration $F_{\bullet} \fu
\Mmod = \fu F_{\bullet} \Mmod$ is again a good filtration. Except for regularity,
this is proved in \cite[Lemme~3.5.5]{Saito-HM}; the point is that 
$\fu \gr_{\bullet}^F \Mmod$ is coherent over $\gr_{\bullet}^F \! \Dmod_Y$, because
pushing forward by finite morphisms preserves coherence.

From now on, we consider the case of a polarizable Hodge module $M \in \HM{X}{w}$. We
say that $f \colon Y \to X$ is non-characteristic for $M$ if it is non-characteristic
for the underlying filtered $\Dmod$-module $(\Mmod, F_{\bullet} \Mmod)$. The
following result shows that the naive inverse image of $M$ is then again a polarizable
Hodge module.

\begin{theorem} \label{thm:inverse-image}
Let $M \in \HMZp{Z}{X}{w}$ be a polarizable Hodge module on $X$, with strict support
$Z$. If $f \colon Y \to X$ is non-characteristic for $M$, then we have
\[
	\fu \Mmod \simeq \Mmod_Y \quad \text{and} \quad
		\fu F_{\bullet} \Mmod \simeq F_{\bullet} \Mmod_Y
\]
for a polarizable Hodge module $M_Y \in \HMZp{f^{-1}(Z)}{Y}{w+r}$.
\end{theorem}

This can be proved in several ways, but perhaps the cleanest one is to use the
relationship between polarizable Hodge modules and polarizable variations of Hodge
structure. According to \cite[Theorem~3.21]{Saito-MHM}, the Hodge module $M$ comes
from a polarizable variation of Hodge structure of weight $w - \dim Z$ on a
Zariski-open subset of the smooth locus of $Z$. We may clearly assume that $f(Y)$
intersects $Z$; the transversality condition implies that the preimage of the smooth
locus of $Z$ is dense in $f^{-1}(Z)$, and that $\dim f^{-1}(Z) = \dim Z + r$. We can
therefore pull the variation of Hodge structure back to a Zariski-open subset of
$f^{-1}(Z)$, and use Saito's result again to extend it to a polarizable Hodge module
$M_Y \in \HMZp{f^{-1}(Z)}{Y}{w+r}$; this procedure explains why the weight changes by
the relative dimension $r$. We denote the underlying filtered $\Dmod$-module by
$(\Mmod_Y, F_{\bullet} \Mmod_Y)$.

It remains to show that $\Mmod_Y \simeq \fu \Mmod$ and that $F_{\bullet} \Mmod_Y
\simeq \fu F_{\bullet} \Mmod$. By construction, this is true on the Zariski-open
subset to which we pulled back the variation of Hodge structure; what we have to
prove is that both sides are extended to $Y$ in the same way. Here the strategy is to
use some of the conditions in the definition of Hodge modules, in particular the
compatibility between the Hodge filtration and the $V\!$-filtration.

We observe first that $f$ factors through its graph as
\[
\begin{tikzcd}
Y \rar{i} \arrow[bend left=40]{rr}{f} & Y \times X \rar{p_2} & X;
\end{tikzcd}
\]
because both $p_2$ and $i$ are again non-characteristic for $M$, it suffices to
deal separately with the case of a smooth morphism and the case of a closed
embedding.

\begin{lemma}
When $f \colon Y \to X$ is a smooth morphism, \theoremref{thm:inverse-image} is true.
\end{lemma}

\begin{proof}
The question is local on $X$, and so we may assume that there is a holomorphic
function $g \colon X \to \CC$ such that $Z_0 = Z \cap g^{-1}(0)$ contains the
singular locus of $Z$, and such that $M$ comes from a polarizable variation of Hodge
structure on $Z \setminus Z_0$. We now define $h = g \circ f$, and consider the
following diagram: 
\[
\begin{tikzcd}
Y \rar{(\id, h)} \dar{f} & Y \times \CC \dar{f \times \id} \\
X \rar{(\id, g)} & X \times \CC 
\end{tikzcd}
\]
Because of \propositionref{prop:Kashiwara}, it suffices to prove the assertion for
the direct image $(\id, g)_{\ast} M$ on $X \times \CC$; this amounts to replacing
$(\Mmod, F_{\bullet} \Mmod)$ by 
\[
	\Mmod_g = \bigoplus_{i \geq 0} \Mmod \tensor \partial_t^i
		\quad \text{and} \quad
		F_{\bullet} \Mmod_g 
		= \bigoplus_{i \geq 0} F_{\bullet-i-1} \Mmod \tensor \partial_t^i,
\]
where $t$ denotes the coordinate on $\CC$ and $\partial_t = \partial/\partial t$
the corresponding vector field.  After making the obvious replacements, we may
therefore assume that $g^{-1}(0)$ and $h^{-1}(0)$ are complex manifolds, and that we
have holomorphic vector fields $\partial_g$ and $\partial_h$ with the property that
$\lie{\partial_g}{g} = 1$ and $\lie{\partial_h}{h} = 1$.

We will first prove that $\Mmod_Y \simeq \fu \Mmod$. Let $V^{\bullet} \Mmod$ and
$V^{\bullet} \Mmod_Y$ denote the rational $V\!$-filtrations along $g = 0$ and $h =
0$, respectively; for left $\Dmod$-modules, the conventions are that $g \colon
V^{\alpha} \Mmod \to V^{\alpha+1} \Mmod$ and $\partial_g \colon V^{\alpha} \Mmod \to
V^{\alpha-1} \Mmod$, and that the operator $g \partial_g - \alpha$ acts nilpotently
on $\gr_V^{\alpha} \Mmod = V^{\alpha} \Mmod / V^{> \alpha} \Mmod$. Now the point is
that $\Mmod$ has strict support $Z$, which is not contained in $g^{-1}(0)$; this
implies that $\partial_g \colon \gr_V^0 \Mmod \to \gr_V^{-1} \Mmod$ is surjective,
and hence that
\[
	\Mmod = \Dmod_X \cdot V^{-1} \Mmod 
	= \sum_{i=0}^{\infty} \partial_g^i \bigl( V^{>-1} \Mmod \bigr).
\]
Recall that $V^{>-1} \Mmod$ only depends on the restriction of $\Mmod$ to $Z
\setminus Z_0$, which is the flat bundle underlying our variation of Hodge structure.
By construction, $V^{>-1} \Mmod_Y \simeq \fu V^{>-1} \Mmod$, and so we obtain
\[
	\fu \Mmod \simeq \sum_{i=0}^{\infty} \partial_h^i \bigl( \fu V^{>-1} \Mmod \bigr)
		\simeq \sum_{i=0}^{\infty} \partial_h^i \bigl( V^{>-1} \Mmod_Y \bigr)
		\simeq \Mmod_Y.
\]
To get the corresponding statement for the filtrations, we will use the fact that $M$ and
$M_Y$ are Hodge modules. One of the conditions in the definition is that the mapping
$\partial_g \colon F_p \gr_V^{\alpha+1} \Mmod \to F_{p+1} \gr_V^{\alpha} \Mmod$ is
surjective for every $p \in \ZZ$ and every $\alpha < -1$; because $M$ has strict
support $Z$, this is also true when $\alpha = -1$. According to
\cite[Remarque~3.2.3]{Saito-HM}, we therefore have
\[
	F_p \Mmod = \sum_{i = 0}^{\infty} 
		\partial_g^i \bigl( V^{>-1} \Mmod \cap \jl \ju F_{p-i} \Mmod \bigr),
\]
where $j \colon X \setminus X_0 \into X$ denotes the open embedding; the right-hand
side is again determined by the variation of Hodge structure on $Z \setminus
Z_0$. Now the flatness condition in the definition of being non-characteristic
implies that
\begin{align*}
	\fu F_p \Mmod &\simeq \sum_{i = 0}^{\infty} 
		\partial_h^i \bigl( \fu V^{>-1} \Mmod \cap \jl \ju \fu F_{p-i} \Mmod \bigr) \\
	&\simeq \sum_{i = 0}^{\infty} 
		\partial_h^i \bigl( V^{>-1} \Mmod_Y \cap \jl \ju F_{p-i} \Mmod_Y \bigr)
	= F_p \Mmod_Y,
\end{align*}
which is the result we were after.
\end{proof}

\begin{lemma}
When $f \colon Y \to X$ is a closed embedding, \theoremref{thm:inverse-image} is
true.
\end{lemma}

\begin{proof}
The problem is again local on $X$, and so we may assume that $f$ is a complete
intersection. If we factor $f$ into a composition of closed embeddings of
codimension $1$, then each step is again non-characteristic by
\cite[Lemme~3.5.4]{Saito-HM}; in this way, we reduce the problem to the case where
$Y$ is defined by a single holomorphic function $g \colon X \to \CC$. Because the
embedding is non-characteristic, it is not hard to show that $V^{\bullet} \Mmod$ is
the $g$-adic filtration \cite[Lemme~3.5.6]{Saito-HM}, and hence that 
\[
	\gr_V^0 \Mmod \simeq \fu \Mmod \quad \text{and} \quad
		\gr_V^{-1} \Mmod \simeq 0;
\]
moreover, the flatness condition implies that $F_{\bullet} \gr_V^0 \Mmod \simeq \fu
F_{\bullet} \Mmod$.  In particular, the action of $N = g \partial_g$ on $\gr_V^0
\Mmod$ is trivial; according to the definition of Hodge modules, this means that the
pair $(\fu \Mmod, \fu F_{\bullet} \Mmod)$ is part of a polarizable Hodge module of
weight $w-1$ on $Y$.  Because $\fu \Mmod$ has strict support $f^{-1}(Z)$, the
uniqueness statement in \cite[Theorem~3.21]{Saito-MHM} implies that this polarizable
Hodge module must be isomorphic to $M_Y$.
\end{proof}

\subsection{Branched coverings}
\label{par:coverings}

The proof of \theoremref{thm:Saito-pure} makes use of certain branched coverings. We
briefly review the construction in the special case that we need; for a more complete
discussion, including proofs, see \cite[\S3]{EV}.

Let $X$ be a complex manifold, and let $L$ be a holomorphic line bundle on $X$.
Suppose that for some integer $N \geq 1$, there is a global section $s \in H^0(X,
L^N)$ whose zero scheme is a smooth divisor $D \subseteq X$. In this situation, one
can construct another complex manifold $Y$ and a branched covering
\[
	\pi \colon Y \to X
\]
in the following way. Let $p \colon L \to X$ denote the projection from the line
bundle, now considered as a complex manifold. Then $\pu L$ has a tautological section
$s_L$, and we may define $Y$ as the zero scheme of the section $s_L^N - \pu s$. It is
easy to see that $Y$ is a complex manifold: over any open subset $U \subseteq X$
where $L$ is trivial, the section $s$ is represented by a holomorphic function $f
\colon U \to \CC$, and $\pi^{-1}(U)$ is the submanifold of $U \times \CC$ defined by
the equation $t^N = f$. The local description can be used to prove that
\[
	\pil \OY \simeq \OX \oplus \bigoplus_{i=1}^{N-1} L^{-i},
\]
and, more generally, that
\[
	\pil \OmY^k \simeq \OmX^k \oplus 
		\bigoplus_{i=1}^{N-1} \OmX^k(\log D) \tensor L^{-i}
\]
for $k = 1, \dotsc, \dim X$; here $\OmX^k(\log D)$ is the sheaf of logarithmic
differential forms. For instance, the summand $L^{-1}$ in the decomposition of $\pil
\OY$ corresponds to $t \cdot \OX$, and the summand $\OmX^k(\log D) \tensor
L^{-1}$ corresponds to 
\[
	t \cdot \OmX^k + \mathit{dt} \wedge \OmX^{k-1}
		= t \cdot \left( \OmX^k + \frac{\df}{f} \wedge \OmX^{k-1} \right),
\]
remembering that $\df/f = N \mathit{dt}/t$. In both formulas, the $N$ summands on the
right-hand side are in one-to-one correspondence with the characters of the group of
$N$-th roots of unity, which acts on $Y$ in the obvious way. 

\section{Proof of the theorem}

\subsection{Extending line bundles}

We now begin with the preparations for the proof of \theoremref{thm:Saito-pure}. 
Let $Z$ be a reduced and irreducible projective variety, and let $L$ be an ample line
bundle on $Z$. Fix an integer $N \geq 2$ such that $L^N$ is very ample; then $Z$
embeds into the projective space $P = \PP \bigl( H^0(Z, L^N) \bigr)$, and the restriction of
$\OP(1)$ is isomorphic to $L^N$. The purpose of this section is to
extend $L$ to a small open neighborhood of $Z$ in $P$; this will allow us to work with
branched coverings that are again complex manifolds.

\begin{lemma} \label{lem:extension}
It is possible to extend $L$ to a holomorphic line bundle $L_X$ on an open
neighborhood $X \supseteq Z$, in such a way that $L_X^N \simeq \OX(1)$.
\end{lemma}

\begin{proof}
According to a result by Durfee \cite[Proposition~1.6 and \S2]{Durfee}, we can find 
an open set $X \subseteq P$ containing $Z$, with the property that the inclusion $Z
\into X$ is a homotopy equivalence. From the exponential sequence -- which is also
valid on $Z$ by definition of the sheaf $\OZ$ -- we obtain a commutative diagram 
\[
\begin{tikzcd}
H^1 \bigl( X, \ZZ(1) \bigr) \arrow[equal]{d} \rar & H^1(X, \OX) \dar \rar{\exp} & 
		H^1(X, \OXst) \dar \rar & H^2 \bigl( X, \ZZ(1) \bigr) \arrow[equal]{d} \\
H^1 \bigl( Z, \ZZ(1) \bigr) \rar & H^1(Z, \OZ) \rar{\exp} & 
		H^1(Z, \OZst) \rar & H^2 \bigl( Z, \ZZ(1) \bigr) 
\end{tikzcd}
\]
with exact rows. By construction, the first Chern class $c_1 \bigl( \OX(1) \bigr) \in
H^2 \bigl( X, \ZZ(1) \bigr)$ maps to $c_1(L^N) = N \cdot c_1(L)$, and is therefore
divisible by $N$. This means that we can find a holomorphic line bundle $M_X$ with
the property that
\[
	N \cdot c_1(M_X) = c_1 \bigl( \OX(1) \bigr) \quad \text{and} \quad
	c_1(M_X) \restr{Z} = c_1(L).
\]
Consequently, there are two elements $\alpha \in H^1(X, \OX)$ and $\beta \in H^1(Z,
\OZ)$ such that
\[
	\exp(\alpha) \cdot \class{M_X}^N = \class{\OX(1)} \quad \text{and} \quad
	\exp(\beta) \cdot \class{M_X} \restr{Z} = \class{L};
\]
square brackets mean the isomorphism class of the corresponding line bundle.
The element $\alpha \restr{Z} - N \beta$ belongs to the image of $H^1 \bigl( Z,
\ZZ(1) \bigr)$; by adjusting $\alpha$, we can arrange that $\beta$ is equal to the
restriction of $\alpha/N$. Now let $L_X$ be any holomorphic line bundle on $X$ with
\[
	\class{L_X} = \exp(\alpha/N) \cdot \class{M_X}.
\]
The formulas above show that $\class{L_X}^N = \class{\OX(1)}$ and $\class{L_X}
\restr{Z} = \class{L}$, and so we have found the desired extension of $L$.
\end{proof}

\subsection{Hodge modules and strictness}

For the remainder of the argument, we may assume that $Z$ is embedded into
a complex manifold $X$, in such a way that the given ample line bundle on $Z$ is the
restriction of a holomorphic line bundle $L$ on $X$. We may also assume that $M \in
\HMZp{Z}{X}{w}$ is a polarizable Hodge module on $X$ with strict support $Z$; this is
because the graded quotients of the de Rham complex do not depend on the embedding
(by \lemmaref{lem:subquotient}). It is important to keep in mind that the underlying
filtered $\Dmod$-module $(\Mmod, F_{\bullet} \Mmod)$ lives on $X$.

Now let $D \subseteq X$ be the divisor of a sufficiently general section $s \in
H^0(X, L^N)$ or, more concretely, the intersection of $X$ with a sufficiently general
hyperplane $H \subseteq P$. Then $D$ is non-characteristic for $M$, and so we obtain
from \theoremref{thm:inverse-image} a polarizable Hodge module $M_D \in \HMZp{D \cap
Z}{D}{w-1}$ with the property that
\[
	\Mmod_D \simeq \Mmod \restr{D} \quad \text{and} \quad
	F_p \Mmod_D \simeq F_p \Mmod \restr{D}.
\]
By construction, we have $L^N \simeq \OX(D)$, and we denote by
\[
	\pi \colon Y \to X
\]
the resulting branched covering of $X$; since $D$ is smooth, $Y$ is again a complex
manifold. It is easy to see that $\pi$ is also non-characteristic for $M$;
this gives us another polarizable Hodge module $M_Y \in \HMZp{\pi^{-1}(Z)}{Y}{w}$
with
\[
	\Mmod_Y \simeq \piu \Mmod \quad \text{and} \quad
	F_p \Mmod_Y \simeq \piu F_p \Mmod.
\]
Both $M_Y$ and $M_D$ will play a role in the proof of \theoremref{thm:Saito-pure}.

We begin by deducing the vanishing of certain morphisms from the fact that $M_Y$ is a
polarizable Hodge module. By construction, the support of $M_Y$ is equal to the
projective variety $\pi^{-1}(Z)$. According to Saito's direct image theorem
\cite[Th\'eor\`eme~5.3.1]{Saito-HM}, the direct image of $(\Mmod_Y, F_{\bullet}
\Mmod_Y)$ under the morphism from $Y$ to a point is therefore strict; concretely,
this means that the spectral sequence 
\begin{equation} \label{eq:spectral-sequence}
	E_1^{p,q} = H^{p+q} \bigl( Y, \gr_{-p}^F \DR(\Mmod_Y) \bigr)
		\Longrightarrow H^{p+q} \bigl( Y, \DR(\Mmod_Y) \bigr)
\end{equation}
degenerates at $E_1$. Since the spectral sequence comes from a filtered complex, it
is easy to describe the $E_1$-differentials in terms of $\DR(\Mmod_Y)$. For each $p \in
\ZZ$, we have a short exact sequence of complexes
\[
	0 \to \gr_{p-1}^F \DR(\Mmod_Y) \to 
		F_p \DR(\Mmod_Y) \big/ F_{p-2} \DR(\Mmod_Y) \to \gr_p^F \DR(\Mmod_Y) \to 0.
\]
In the derived category of complexes of sheaves of $\CC$-vector spaces, it is part of
a distinguished triangle; the third morphism in this triangle is
\[
	\gr_p^F \DR(\Mmod_Y) \to \gr_{p-1}^F \DR(\Mmod_Y) \decal{1}.
\]
As in the case of $d \colon \OY \to \OmY^1$, this morphism is in general not
$\OY$-linear. The degeneration of the spectral sequence in
\eqref{eq:spectral-sequence} has the following consequence.

\begin{lemma} \label{lem:strictness}
For every $i,p \in \ZZ$, the induced morphism on cohomology
\[
	H^i \bigl( Y, \gr_p^F \DR(\Mmod_Y) \bigr)
		 \to H^{i+1} \bigl( Y, \gr_{p-1}^F \DR(\Mmod_Y) \bigr)
\]
is equal to zero.
\end{lemma}

\subsection{Comparison with the original complex}

The purpose of this section is to obtain information about the complex $\gr_p^F
\DR(\Mmod)$ from \lemmaref{lem:strictness}. The first step is to take the direct
image of $\DR(\Mmod_Y)$ by the finite morphism $\pi \colon Y \to X$.  

\begin{lemma} \label{lem:subcomplex}
The complex $\pil \DR(\Mmod_Y)$ has a direct summand isomorphic to
\[
	\Bigl\lbrack
		L^{-1} \tensor \Mmod \to \OmX^1(\log D) \tensor L^{-1} \tensor \Mmod \to \dotsb 
			\to \OmX^n(\log D) \tensor L^{-1} \tensor \Mmod
	\Bigr\rbrack \decal{n},
\]
compatible with the filtration $\pil F_{\bullet} \DR(\Mmod_Y)$.
\end{lemma}

\begin{proof}
Note that the functor $\pil$ is exact because $\pi$ is a finite morphism; the
isomorphism $\Mmod_Y \simeq \piu \Mmod$ and the projection formula therefore imply that
\[
	\pil \DR(\Mmod_Y) \simeq \Bigl\lbrack
		\pil \OY \tensor \Mmod \to \pil \OmY^1 \tensor \Mmod \to \dotsb 
			\to \pil \OmY^n \tensor \Mmod
	\Bigr\rbrack \decal{n}.
\]
We now take the summand with $L^{-1}$ in the decomposition of each term (see
\parref{par:coverings}). To show that
this leads to a subcomplex, we can exploit the group action: the group of $N$-th roots of
unity acts on the entire complex, and we are taking the summand corresponding to the
standard character. That the decomposition respects the filtration is obvious. 
\end{proof}

We shall give a second proof in local coordinates in \parref{par:computations} below.
To simplify the notation, let us denote by
\[
	\Kt \subseteq \pil \DR(\Mmod_Y) \quad \text{and} \quad
		F_{\bullet} \Kt \subseteq \pil F_{\bullet} \DR(\Mmod_Y)
\]
the subcomplex in \lemmaref{lem:subcomplex}, together with the induced filtration. As
before, we have a collection of $\CC$-linear connecting morphisms
\[
	\deltat_p \colon \gr_p^F \Kt \to \gr_{p-1}^F \Kt \decal{1}
\]
in the derived category; because $\Kt$ is a direct summand, the degeneration of the
spectral sequence in \eqref{eq:spectral-sequence} means that the induced morphisms
\[
	H^i \bigl( X, \gr_p^F \Kt \bigr)
		 \to H^{i+1} \bigl( X, \gr_{p-1}^F \Kt \bigr)
\]
are also equal to zero. To exploit this fact, we are now going to relate the graded
quotients $\gr_p^F \Kt$ to the two complexes $\gr_p^F \DR(\Mmod)$ and $\gr_p^F
\DR(\Mmod_D)$.

\begin{proposition} \label{prop:morphisms}
Let $L_D$ denote the restriction of $L$ to the divisor $D$. 
\begin{aenumerate}
\item We have a morphism of complexes
\[
	f_p \colon L^{-1} \tensor \gr_p^F \DR(\Mmod) \to \gr_p^F \Kt,
\]
induced by the natural inclusions $\OmX^k \into \OmX^k(\log D)$.
\item We have a morphism of complexes
\[
	r_p \colon \gr_p^F \Kt \to L_D^{-1} \tensor \gr_{p+1}^F \DR(\Mmod_D),
\]
induced by the residue mappings $\Res_D \colon \OmX^k(\log D) \to \OmD^{k-1}$.
\end{aenumerate}
\end{proposition}

The proof requires a small calculation in local coordinates; we postpone it until
\parref{par:computations} and first state the main result. 

\begin{proposition} \label{prop:composition}
Up to a constant factor of $(-1)^n N$, the composition
\[
\begin{tikzcd}
	L^{-1} \tensor \gr_p^F \DR(\Mmod) \rar{f_p} & \gr_p^F \Kt \rar{\deltat_p} & 
		\gr_{p-1}^F \Kt \decal{1} \rar{r_{p-1}} & 
		L_D^{-1} \tensor \gr_p^F \DR(\Mmod_D) \decal{1}
\end{tikzcd}
\]
is equal to the restriction mapping.
\end{proposition}

The proof can be found in \parref{par:computations}. The point is of course
that, because of the above factorization, the induced morphism on cohomology
\begin{equation} \label{eq:restriction}
	H^i \Bigl( X, L^{-1} \tensor \gr_p^F \DR(\Mmod) \Bigr)
	\to H^{i+1} \Bigl( D, L_D^{-1} \tensor \gr_p^F \DR(\Mmod_D) \Bigr)
\end{equation}
is equal to zero. Once this is known, \theoremref{thm:Saito-pure} can be proved very
easily by using Serre's vanishing theorem and induction on the dimension.

\subsection{Proof of Saito's theorem}

We are now ready to prove \theoremref{thm:Saito-pure}. We first observe that the two
assertions 
\begin{align}
	H^i \bigl( Z, \gr_p^F \DR(\Mmod) \tensor L \bigr) = 0 
		\quad \text{for $i > 0$ and $p \in \ZZ$,} \label{eq:assertion-1} \\
	H^i \bigl( Z, \gr_p^F \DR(\Mmod) \tensor L^{-1} \bigr) = 0 
		\quad \text{for $i < 0$ and $p \in \ZZ$,} \label{eq:assertion-2}
\end{align}
are equivalent to each other by virtue of \lemmaref{lem:polarization}; it is
therefore enough to prove the second one. This will be done by induction on the
dimension. Since $D \subseteq X$ is a smooth divisor with $L^N \simeq \OX(D)$, we
have a short exact sequence 
\[
	0 \to L_D^{-N} \to \OmX^1 \restr{D} \to \OmD^1 \to 0.
\]
As shown in \parref{par:computations} below, it induces a short exact sequence of
complexes
\begin{equation} \label{eq:short-exact}
	0 \to L_D^{-N} \tensor \gr_{p+1}^F \DR(\Mmod_D) \to
		\gr_p^F \DR(\Mmod) \restr{D} \to \gr_p^F \DR(\Mmod_D) \decal{1} \to 0.
\end{equation}
By induction, we can assume that the $i$-th cohomology of $L_D^{-N-1} \tensor
\gr_{p+1}^F \DR(\Mmod_D)$ vanishes for every $i < 0$; it follows that
\begin{equation} \label{eq:induction}
	H^i \Bigl( D, L_D^{-1} \tensor \gr_p^F \DR(\Mmod) \restr{D} \Bigr)
	\to H^{i+1} \Bigl( D, L_D^{-1} \tensor \gr_p^F \DR(\Mmod_D) \Bigr)
\end{equation}
is injective. Because we already know that the morphism in \eqref{eq:restriction} is
equal to zero, the injectivity of \eqref{eq:induction} means that the morphism
\[
	H^i \Bigl( X, L^{-1} \tensor \gr_p^F \DR(\Mmod) \Bigr)
	\to H^i \Bigl( D, L_D^{-1} \tensor \gr_p^F \DR(\Mmod_D) \restr{D} \Bigr)
\]
is also equal to zero. This obviously implies the surjectivity of
\[
	H^i \Bigl( X, L^{-N-1} \tensor \gr_p^F \DR(\Mmod) \Bigr) \to
	H^i \Bigl( X, L^{-1} \tensor \gr_p^F \DR(\Mmod) \Bigr)
\]
for $i < 0$; note that the morphism is nothing but multiplication by the 
global section $s \in H^0(X, L^N)$ that we chose at the beginning of the proof.

Now we can easily complete the proof of \theoremref{thm:Saito-pure} with the help of
Serre's vanishing theorem. Recall that $\gr_p^F \DR(\Mmod) \in \Dbcoh(\OZ)$ does not
depend on the choice of embedding; what we have shown above is that the
multiplication morphism
\begin{equation} \label{eq:multiplication}
	H^i \Bigl( Z, L^{-N-1} \tensor \gr_p^F \DR(\Mmod) \Bigr) \to
	H^i \Bigl( Z, L^{-1} \tensor \gr_p^F \DR(\Mmod) \Bigr)
\end{equation}
is surjective for every $i < 0$ and every sufficiently general section $s \in H^0(Z, L^N)$.
By \lemmaref{lem:polarization} and Serre duality, we have
\[
	H^i \Bigl( Z, L^{-N-1} \tensor \gr_p^F \DR(\Mmod) \Bigr)
		\simeq H^{-i} \Bigl( Z, L^{N+1} \tensor \gr_{-p-w}^F \DR(\Mmod) \Bigr).
\]
This becomes equal to zero for $N \gg 0$, because $\gr_{-p-w}^F \DR(\Mmod)$ is
concentrated in non-positive degrees. The surjectivity of \eqref{eq:multiplication}
therefore implies the desired vanishing for $\gr_p^F \DR(\Mmod)$.

\begin{note}
The proof becomes simpler in the case of the lowest graded quotient
\[
	\gr_{p(M)}^F \DR(\Mmod) = \omX \tensor F_{p(M) + n} \Mmod
\]
of the de Rham complex. This amounts to taking $p = -p(M)-w$ in the argument
above; the point is that the complex $\gr_{p+1}^F \DR(\Mmod_D)$ is now exact, because
\[
	p + 1 = 1 - p(M) - w = 1 - p(M_D) - (w - 1).
\]
Consequently, \eqref{eq:induction} is automatically injective, and so we do not need
any vanishing on $D$ to conclude that \eqref{eq:multiplication} is surjective.
Many other interesting results about the coherent $\OZ$-module $\gr_{p(M)}^F
\DR(\Mmod)$ can be found in \cite{Saito-K}.
\end{note}

\subsection{Computations in local coordinates}
\label{par:computations}

We now prove \propositionref{prop:morphisms} and \propositionref{prop:composition},
as well as the exactness of the sequence of complexes in \eqref{eq:short-exact}.
Since it is easiest to do this by a calculation in local coordinates, we shall first
give a description of the complex $\Kt$ in a neighborhood of the divisor $D$.

Let $x_1, \dotsc, x_n$ be local holomorphic coordinates on $X$, with the property
that the divisor $D$ is defined by the equation $x_n = 0$. On $Y$, we can choose
local holomorphic coordinates $y_1, \dotsc, y_n$ in such a way that $\pi \colon Y
\to X$ is represented by
\[
	(x_1, \dotsc, x_{n-1}, x_n) = (y_1, \dotsc, y_{n-1}, y_n^N).
\]
In particular, the line bundle $L$ is trivial on the open set in question; note that
the summand $L^{-1}$ in the decomposition of $\pil \OY$ corresponds to $y_n \cdot \OX$. 

The $\Dmod$-module structure on $\Mmod_Y \simeq \piu \Mmod$ comes from the natural
morphism $\Dmod_Y \to \piu \Dmod_X$. As in \parref{par:DR}, the differentials
$\nabla_Y \colon \OmY^k \tensor \Mmod_Y \to \OmY^{k+1} \tensor \Mmod_Y$ in the de
Rham complex of $\Mmod_Y$ are therefore given in local coordinates by 
\[
	\nabla_Y(\alpha \tensor m) = (-1)^n d\alpha \tensor m + (-1)^n 
		\sum_{i=1}^n (\piu \dx_i \wedge \alpha) \tensor \frac{\partial}{\partial x_i} m.
\]
Since $L$ is trivial on the open set in question, the induced differentials
\[
	\pil \nabla_Y \colon \OmX^k(\log D) \tensor \Mmod 
		\to \OmX^{k+1}(\log D) \tensor \Mmod
\]
are represented by the formula
\begin{equation} \label{eq:new-differential}
\begin{split}
	\pil \nabla_Y(\alpha \tensor m) &= (-1)^n \frac{d(y_n \alpha)}{y_n} \tensor m + (-1)^n 
		\sum_{i=1}^n (\dx_i \wedge \alpha) \tensor \frac{\partial}{\partial x_i} m \\
	&= (-1)^n \frac{1}{N} \frac{\dx_n}{x_n} \wedge \alpha \tensor m 
		+ \nabla(\alpha \tensor m),
\end{split}
\end{equation}
where $\nabla$ is defined as in \eqref{eq:DR-local}. With this description,
\propositionref{prop:morphisms} is easy.

\begin{proof}[Proof of \propositionref{prop:morphisms}]
The formula in \eqref{eq:new-differential} shows that the morphisms
\[
	\OmX^k \tensor L^{-1} \tensor \gr_{p+k}^F \Mmod \to 
		\OmX^k(\log D) \tensor L^{-1} \tensor \gr_{p+k}^F \Mmod
\]
are compatible with the differentials in $L^{-1} \tensor \gr_p^F \DR(\Mmod)$ and
$\gr_p^F \Kt$; this proves the first assertion. Our definition of the residue mapping
is
\[
	\Res_D \colon \OmX^k(\log D) \to \OmD^{k-1}, \quad
		\Res_D \left( \frac{df}{f} \wedge \alpha \right) = \alpha \restr{D}
\]
where $f$ is an arbitrary local defining equation for $D$; it interacts better with
the sign conventions for the de Rham complex than the usual definition. The residue
mapping induces morphisms 
\[
	r_p \colon \OmX^k(\log D) \tensor L^{-1} \tensor \gr_{p+k}^F \Mmod
		\to \OmD^{k-1} \tensor L_D^{-1} \tensor \gr_{p+k}^F \Mmod_D,
\]
and we have to check that they are compatible with the differentials in both
complexes. This is straightforward: the residue of
\[
	(-1)^n \sum_{i=1}^n (\dx_i \wedge \alpha) \tensor \frac{\partial}{\partial x_i} m
\]
is evidently 
\[
	(-1)^n \sum_{i=1}^n \Res_D (\dx_i \wedge \alpha) \tensor 
		\frac{\partial}{\partial x_i} m \restr{D}
	= (-1)^{n-1} \sum_{i=1}^{n-1} \dx_i \wedge \Res_D(\alpha) \tensor
		\frac{\partial}{\partial x_i} m \restr{D},
\]
hence equal to what we get when we apply $\nabla_D$ to $\Res_D(\alpha) \tensor m \restr{D}$.
\end{proof}

Now we can prove the main technical result, namely \propositionref{prop:composition}.

\begin{proof}[Proof of \propositionref{prop:composition}]
It suffices to check this in a neighborhood of any given point of $D$. After choosing 
local coordinates as above, the line bundle $L$ becomes trivial on the open set in
question, and the differentials in the complex $\Kt$ are given by the formula in
\eqref{eq:new-differential}. To simplify the notation, we define
\[
	K = \DR(\Mmod) \quad \text{and} \quad
		F_{\bullet} K = F_{\bullet} \DR(\Mmod),
\]
and denote by $\delta_p \colon \gr_p^F K \to \gr_{p-1}^F K \decal{1}$ the connecting
morphisms in the derived category. Since we have trivialized $L$, the morphisms
$\OmX^k \tensor \Mmod \to \OmX^k(\log D) \tensor \Mmod$ give rise to a
commutative diagram 
\[
\begin{tikzcd}
0 \rar & \gr_{p-1}^F K \rar \dar{f_{p-1}} & F_p K / F_{p-2} K \rar \dar[dashed] & 
		\gr_p^F K \rar \dar{f_p} & 0  \\
0 \rar & \gr_{p-1}^F \Kt \rar & F_p \Kt / F_{p-2} \Kt \rar & \gr_p^F \Kt \rar & 0
\end{tikzcd}
\]
where both rows are exact and the solid arrows are morphisms of complexes (by
\propositionref{prop:morphisms}); the dashed arrow is not a morphism of complexes. In
the derived category, we now consider the following square, which is in general
not commutative:
\[
\begin{tikzcd}
\gr_p^F K \rar{\delta_p} \dar{f_p} & \gr_{p-1}^F K \decal{1} \dar{f_{p-1}} \\
\gr_p^F \Kt \rar{\deltat_p} & \gr_{p-1}^F \Kt \decal{1}
\end{tikzcd}
\]
According to \lemmaref{lem:connecting} below, the difference
\[
	\deltat_p f_p - f_{p-1} \delta_p \colon \gr_p^F K \to \gr_{p-1}^F \Kt \decal{1}
\]
can be computed by comparing the differential in $\Kt$ and the differential in $K$;
going back to \eqref{eq:new-differential}, the result is that $\deltat_p f_p -
f_{p-1} \delta_p$ equals
\[
	\OmX^k \tensor \gr_{p+k}^F \Mmod \to \OmX^{k+1}(\log D) \tensor \gr_{p+k}^F \Mmod,
		\quad 
	\alpha \tensor m \mapsto (-1)^n \frac{1}{N} \frac{\dx_n}{x_n} \wedge \alpha \tensor m.
\]
If we compose this with the residue mapping, we find that the morphism $r_{p-1}
\deltat_p f_p = r_{p-1}(\deltat_p f_p - f_{p-1} \delta_p)$ is equal to
\[
	\OmX^k \tensor \gr_{p+k}^F \Mmod \to \OmD^k \tensor \gr_{p+k}^F \Mmod_D,
		\quad 
	\alpha \tensor m \mapsto (-1)^n \frac{1}{N} \alpha \restr{D} \tensor m \restr{D},
\]
and therefore agrees with the restriction mapping up to a factor of $(-1)^n N$. 
\end{proof}

It remains to say a few words about the sequence of complexes in
\eqref{eq:short-exact}. Starting from the short exact sequence of locally free
sheaves
\[
	0 \to L_D^{-N} \to \OmX^1 \restr{D} \to \OmD^1 \to 0,
\]
we can take exterior powers to obtain a family of short exact sequences
\[
	0 \to L_D^{-N} \tensor \OmD^{k-1} \to \OmX^k \restr{D} \to \OmD^k \to 0
\]
for $k = 0, 1, \dotsc, n$. Since $\OmD^k$ is locally free, the resulting sequences
\[
	0 \to L_D^{-N} \tensor \OmD^{k-1} \tensor \gr_{p+k}^F \Mmod_D
		\to \bigl( \OmX^k \tensor \gr_{p+k}^F \Mmod \bigr) \restr{D} 
		\to \OmD^k \tensor \gr_{p+k}^F \Mmod_D \to 0
\]
are still short exact. Using the formulas in \parref{par:DR}, it is an easy exercise to
show that both morphisms are compatible with the differentials.

\subsection{Connecting morphisms}

In this section, we prove a small lemma about connecting morphisms in short exact
sequences of complexes. Let $B_1$ and $B_2$ be two complexes in an abelian category,
and suppose that we have a family of morphisms
\[
	f^n \colon B_1^n \to B_2^n
\]
that do not necessarily commute with the differentials; this is of course precisely
the situation that we encountered during the proof of \propositionref{prop:composition}. 
If we define
\[
	\varphi^n = f^{n+1} d_1^n - d_2^n  f^n 
		\colon B_1^n \to B_2^{n+1},
\]
then $\varphi \colon B_1 \to B_2 \decal{1}$ is a morphism of complexes; here it is
necessary to remember that the $n$-th differential in the shifted complex $B_2
\decal{1}$ is equal to $-d_2^{n+1}$. Suppose in addition that we have the
following commutative diagram: 
\[
\begin{tikzcd}
0 \rar & A_1 \rar{i_1} \dar{e} & B_1 \rar{p_1} \dar[dashed]{f} & C_1 \rar \dar{g} & 0  \\
0 \rar & A_2 \rar{i_2} & B_2 \rar{p_2} & C_2 \rar & 0 
\end{tikzcd}
\]
In this diagram, all solid arrows are morphisms of complexes; both squares commute;
and both rows are exact. Because $e$ and $g$ commute with the differentials, it is
easy to see that $\varphi = i_2  \psi  p_1$ for a unique morphism of complexes $\psi
\colon C_1 \to A_2 \decal{1}$. Note that although $\varphi$ is homotopy equivalent to
zero, this is no longer the case for $\psi$; in particular, viewed as a morphism in
the derived category, $\psi$ is typically nonzero.

In the derived category, each row of the diagram is part of a distinguished triangle,
and we denote the third morphism in this triangle by $\delta_k \colon C_k \to A_k
\decal{1}$. We can now consider the following square of morphisms:
\[
\begin{tikzcd}
C_1 \rar{\delta_1} \dar{g} & A_1 \decal{1} \dar{e} \\
C_2 \rar{\delta_2} & A_2 \decal{1}
\end{tikzcd}
\]
Unless $f$ is a morphism of complexes, the square is not commutative. The following
lemma shows that 

\begin{lemma} \label{lem:connecting}
In the derived category, we have $\delta_2  g - e  \delta_1 = \psi$.
\end{lemma}

\begin{proof}
We begin by describing the morphism $\delta_k$. Let 
\[
	M_k = B_k \oplus A_k \decal{1}
\]
denote the mapping cone of $i_k \colon A_k \to B_k$, with differential given by the
matrix
\[
	\begin{pmatrix}
		d_k & i_k \\
		0 & -d_k
	\end{pmatrix}.
\]
There are two obvious morphisms $p_k \colon M_k \to C_k$ and $q_k \colon M_k \to A_k
\decal{1}$; the first one is a quasi-isomorphism, and $\delta_k p_k =
q_k$. Next, we observe that 
\[
	\begin{pmatrix}
		d_2 & i_2 \\
		0 & -d_2
	\end{pmatrix} \begin{pmatrix}
		f & 0 \\
		\psi p_1 & e
	\end{pmatrix} = \begin{pmatrix}
		f & 0 \\
		\psi p_1 & e
	\end{pmatrix} \begin{pmatrix}
		d_1 & i_1 \\
		0 & -d_1
	\end{pmatrix},
\]
which means exactly that 
\[
	h = \begin{pmatrix}
		f & 0 \\
		\psi p_1 & e
	\end{pmatrix} \colon M_1 \to M_2
\]
is a morphism of complexes with $p_2 h = g p_1$ and $q_2 h = \psi p_1 + e q_1$. But then
\[
	(\delta_2 g - e \delta_1) p_1 = \delta_2 p_2 h - e q_1 = q_2 h - e q_1
		= \psi p_1,
\]
which proves the assertion because $p_1$ is a quasi-isomorphism.
\end{proof}


\bibliographystyle{amsalphax}
\providecommand{\bysame}{\leavevmode\hbox to3em{\hrulefill}\thinspace}
\providecommand{\ZM}{\relax\ifhmode\unskip\space\fi Zbl }
\providecommand{\MR}{\relax\ifhmode\unskip\space\fi MR }
\providecommand{\arXiv}[1]{\relax\ifhmode\unskip\space\fi\href{http://arxiv.org/abs/#1}{arXiv:#1}}
\providecommand{\MRhref}[2]{%
  \href{http://www.ams.org/mathscinet-getitem?mr=#1}{#2}
}
\providecommand{\href}[2]{#2}

\end{document}